\documentclass[a4paper,10pt]{amsart}
\usepackage[utf8]{inputenc}
\usepackage{amsmath, amssymb, amsthm, array, marginnote, xcolor, footnote, enumerate, enumitem, MnSymbol, bbm, mathtools, multirow, arydshln, hyperref}
\usepackage[
backend=biber,
style=alphabetic,
sorting=nyt,
giveninits=true,
maxbibnames=9
]{biblatex}
\usepackage{svg, float, subcaption, bbm}
\usepackage{tikz, tikz-cd, circuitikz}
\usetikzlibrary{matrix, arrows.meta, trees, positioning}

\usepackage{xcolor}
\textheight9in  \def\DATE{\today}
\hoffset - 2cm  \hbadness=100000

\def\S{{\mathbb S}}
\def\Z{{\mathbb Z}}
\def\P{{\mathcal{P}}}
\def\Q{{\mathcal{Q}}}
\def\M{{\mathcal{M}}}

\theoremstyle{definition}

\theoremstyle{plain}
{
\newtheorem{lemma}{Lemma}
\newtheorem{proposition}{Proposition}
\newtheorem{theorem}{Theorem}
}
\theoremstyle{definition}
{
\newtheorem{example}{Example}
\newtheorem{corollary}{Corollary}
\newtheorem{remark}{Remark}
}
\newtheorem*{theorem*}{Theorem}

\title{Percolating sets and the operad of permutations}
\date{\today}


\author[D. Bashkirov]{Denis Bashkirov} 
\address{Czech Academy of Sciences, Institute of Mathematics, {\v Z}itn{\'a} ,25,
         115 67 Prague 1, Czech Republic}

\begin{document}
\begin{abstract}
 We give an operadic interpretation of the known result of L.Shapiro and A.B.Stephens that characterizes percolating permutation matrices. A relation of ideals and suboperads of the non-symmetric operad of permutations to percolative properties of sets in the 2-neighbor percolation process on a square grid is discussed. On a related note, we discuss a certain presentation of the operad of permutations.
\end{abstract}

\maketitle

\section{Introduction}
Recall the setting of a 2-neighbor bootstrap percolation process on a rectangular $n$-by-$m$ lattice. It is a discrete dynamical system (a cellular automaton) defined by the following data:
\begin{enumerate}
\item 
an \emph{initial configuration}: a $n$-by-$m$ grid, where each cell is assigned one of the two possible states (a \emph{color}) - "red" or "blue".
\item 
an \textit{update rule}: a blue cell of the grid will change its color to red if at least two of its directly adjacent neighboring cells are red. Here, \emph{directly adjacent} means being aligned either horizontally or vertically. \item 
the \textit{dynamics}: the colors of the cells change iteratively starting with the initial configuration. 
At each iteration, all the blue cells that satisfy the condition of having at least two red neighbors turn red simultaneously. A red cell remains red for the rest of the iterative process. The process continues until no more blue cells can change their color. The stable state that the grid eventually reaches will be referred to as its \textit{final configuration}.
\end{enumerate}

\begin{figure}[H]
    \centering
    \begin{tabular}{ccc}
       \includegraphics[width=2.2cm]{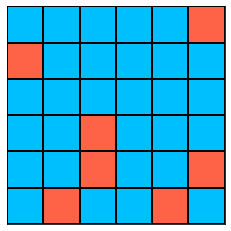}  &  
       \includegraphics[width=2.2cm]{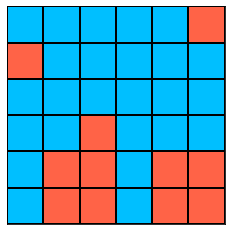} &
       \includegraphics[width=2.2cm]{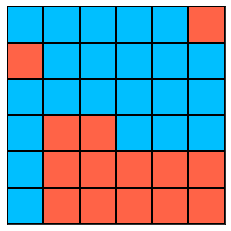}\\
    \end{tabular}
    \caption{Three consecutive steps of the percolation process}    
\end{figure}
In the context of the Ising model, this discrete dynamical system arises as the zero-temperature regime of the isotropic ferromagnetic Glauber dynamics on a $n$-by-$m$ rectangular lattice with the free boundary conditions in presence of a (strong enough) external transverse magnetic field. The colors correspond to the spin states and red is the one that is aligned with the external field. A grid, in the sense of our terminology, is the Poincar{\'e}-dual depiction of a rectangular lattice graph, where vertices are replaced by 2-cells and edges are flipped to the transverse direction. We assume that grids come equipped with the standard matrix-like coordinate system in terms of rows and columns enumerated from top to bottom and from left to right respectively.

An initial configuration is said to be \emph{percolating} if the final configuration of the grid consists of red cells only, thus yielding the lowest energy state of the Ising Hamiltonian. A \emph{percolating set} $\sigma$ is the set of all red cell positions $(i,j)$ in a percolating configuration. A set of positions $\tau$ is said to be \emph{non-percolating} if the initial configuration has red cells positioned at all $(i,j)\in \tau$, but the final configuration contains at least one blue cell. We will routinely identify a grid configuration, percolating or not, with the corresponding set of the red cell positions. 
For a given initial configuration $\sigma$, the corresponding final configuration will be denoted by $\overline{\sigma}$. We say that a set $\sigma$ \textit{spans} a region $R$ in a grid if $\overline{\sigma}\cap R=R$.

Percolating sets are partially ordered by inclusion. A percolating set $\sigma$ is said to be \textit{minimal} if $\sigma\setminus\{x\}$ is a non-percolating set for any $x\in\sigma$. One can show that any percolating set $\sigma$ in a square $n$-by-$n$ grid is of size at least $n$. The lower bound is sharp as attested, for instance, by a diagonal configuration. 
The following are the known values for the number of all minimal percolating sets of size $n\geq 1$ in a $n$-by-$n$ grid: 
\[
1, 2, 14, 130, 1615, 23140, 383820, 7006916, 140537609, 3035127766.
\]
This is sequence A145901\cite{oeisA145901}.
A result of R.Morris\cite{Morris} provides an upper bound for the size of a minimal percolating set $\sigma$ on a square $n$-by-$n$ grid as $|\sigma|\leq \frac{(n+2)^2}{6}$.

To better understand the dynamics of the model, it is of interest to classify percolating and non-percolating sets, the minimal and the maximal ones respectively, in particular.
This includes identifying non-trivial infinite families of both percolating and non-percolating sets. The reader will find an example of such a family of minimal percolating sets in \cite[Lemma 3]{Morris}. A simpler example from the same work is the family of minimal percolating sets $A(m)$ of size $|A(m)|=4m$ on a $3m$-by-$3m$ grid, defined by the red cell placements at
$(1,3k-1)$, $(1,3k)$, $(3k-1,1)$, $(3k,1)$ for all $1\leq k\leq m$.
\begin{figure}[H]
\centering
\includegraphics[width=2.5cm]{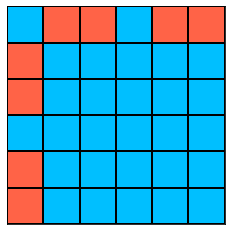}
\caption{$A(2)$}
\end{figure}

One way to generate families of percolating and non-percolating sets is to invoke the block renormalization procedure. Namely, given a configuration $\sigma$ in a $n$-by-$n$ grid and $k>1$, a new configuration is created by fine-graining the original grid by a factor of $k$, that is, by dissecting each row and each column of the grid into $k$ rows and $k$ columns respectively. Under such a procedure, a cell gets replaced by a $k$-by-$k$ configuration with all cells of the same respective color. Furthermore, we may reduce a fine-grained configuration by replacing any all-red $k$-by-$k$ block by an arbitrary percolating $k$-by-$k$ configuration instead. Jointly, all this can be regarded as a compound operation on $\sigma$ that amounts to replacing each red cell of $\sigma$ by a percolating $k$-by-$k$ configuration and replacing each blue cell by an all-blue $k$-by-$k$ block.
\begin{table}[H]
\centering
\begin{tabular}{ccc}
\resizebox{2.3cm}{!}{\includegraphics{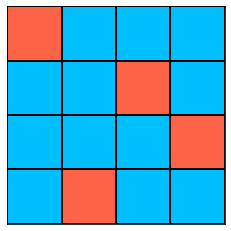}}     & 
\resizebox{2.3cm}{!}{\includegraphics{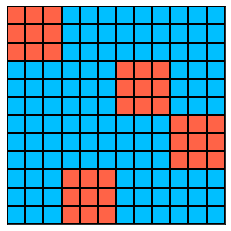}}
   & 
\resizebox{2.3cm}{!}{\includegraphics{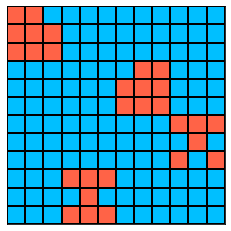}}
\end{tabular}
\captionof{figure}{Fine-graining and a reduction of a configuration.}
\end{table}
A natural question to ask is whether such an operation applied to a configuration $\sigma$ preserves its property of being (non-)percolating. In general, the answer is negative as may be observed, for instance, in the following example, where block renormalization by a factor of $k=2$ transforms a percolating configuration into a non-percolating one.
\begin{table}[H]
\centering
\begin{tabular}{cc}
\resizebox{2.3cm}{!} {\includegraphics{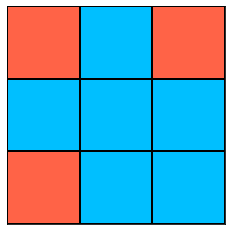}} & 
\resizebox{2.3cm}{!}{\includegraphics{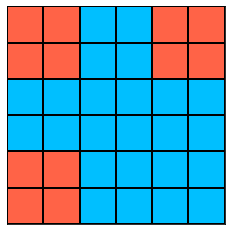}}
\end{tabular}
\captionof{figure}{A percolating configuration and its non-percolating fine-grained image.}
\end{table}
Nevertheless, as we discuss in the present work, there are infinite families of percolating and non-percolating sets that do retain their respective percolative qualities under block renormalization and reduction. Such families
can be generated by means of a certain operation that is associative, or generalized associative, in an appropriate sense.
A complementary point of view is that such families tend to be \emph{operadic}, meaning they are closed under a certain generalized-associative operation.
In this regard, the infinite family of minimal percolating sets examined by L. Shapiro and A. B. Stephens in \cite{ShapiroStephens} is of particular interest for us.
The family is defined in the following way. Let $\sigma$ be a $n$-by-$n$ permutation matrix regarded as a grid configuration upon identifying $1$'s with red cells, and $0$'s with blue ones. As shown in the work cited, such a configuration is percolating if and only if the permutation $\sigma$ is separable. We recall that a permutation is said to be \emph{separable}, if it can be built recursively, starting with the trivial permutation of length $1$, using two operations: the direct sum $\sigma\oplus \tau$ and the skew sum $\sigma\ominus\tau$ of two permutations. In terms of the corresponding permutation matrices these operations can be defined by forming a block matrix out of the two given permutation matrices in two different ways:

\begin{align*}
\sigma \oplus \tau=
    \left[
    \begin{array}{c;{2pt/2pt}c}
        \sigma & \Huge 0 \\ \hdashline[2pt/2pt]
        \Huge 0 & \tau 
    \end{array}
\right],\quad 
\sigma \ominus \tau=
    \left[
    \begin{array}{c;{2pt/2pt}c}
        \Huge 0 & \sigma  \\ \hdashline[2pt/2pt]
        \tau & \Huge 0
    \end{array}
\right].
\end{align*}
Separable permutations are counted by the big Schr\"oder numbers (A006318 \cite{oeisA006318}):
\[
1, 2, 6, 22, 90, 394, 1806, 8558, 41586, 206098\,\dots
\]
Other families of objects counted by this sequence include the combinatorial types of paper guillotine cuts, plane two-terminal (or bipolar) series-parallel networks and bicolored trees.
In operadic terms, the result of Shapiro and Stephens can be restated as follows.
\begin{theorem*}[\cite{ShapiroStephens}]
The percolating $n$-by-$n$ permutation matrices for $n\geq 2$ form a suboperad $Perm_2$ of the non-$\Sigma$ operad of permutations $Perm$. The suboperad is generated by permutations $12$ and $21$ in arity $2$ and consists of all separable permutations. 
\end{theorem*}
A reader familiar with the notion of a \emph{non-$\Sigma$} (also called \emph{nonsymmetric}) operad of permutations may proceed directly to section \ref{percsec} for a proof. Otherwise, a brief introduction to the subject is provided in section~ \ref{backopsec} below. In the remainder of the paper we discuss a presentation of the non-$\Sigma$ operad of $Perm$. Note that while $Perm$ can be regarded as the desymmetrization of the symmetric associative operad $\mathcal{A}s$ in sets, it is, unlike $\mathcal{A}s$, not binary-generated, it requires an infinite countable set of generators, and is a fairly transcendental object, since the generating function for the cardinalities of the canonical generators, as we discuss in section \ref{permpres}, is known not be non-algebraic \cite{AlbertAtkinsonKlazar}.
 \begin{theorem*}
$Perm$ admits a presentation with generators $G_n$ indexed by oriented Hamiltonian cycles on the complement of the cyclic graph $C_{n+1}$ for $n\geq 4$, and two exceptional generators $G_2=\{12,21\}$ in arity $n=2$.
\end{theorem*}
The generators of this presentation can be related to non-percolating permutation matrices.

\section{A background on operads}
\label{backopsec}
For our purposes, we need only a bare minimum of operadic algebra and only three particular examples of \emph{non-$\Sigma$ operads}: the operad of permutations $Perm$, the associative operad $As$, and the \emph{free product} $As*As$ of $As$ with itself.
All operads that we consider are defined in the usual symmetric monoidal category of sets. A reader unfamiliar with the general definition of an operad might regard constructions presented below as ad hoc examples of a monoid-like algebraic structure defined on the disjoint union 
${\P=\bigsqcup\limits_{n\geq 1}\P(n)}$ of some sets $\P(n)$ by means of a countable family of binary operations, called \emph{partial compositions} that are of the form $(-\circ_i-):\P(m)\times \P(n)\to \P(m+n-1)$ for all $m,n\geq 1$ and $1\leq i\leq m$. The non-triviality and utility of the notion is due to the fact that partial compositions collectively satisfy a certain property, a form of generalized associativity, which is a defining feature of operads. Namely, 
for any~$k,l,m\geq 1$, $a\in \P(k)$, $b\in \P(l)$, $c\in\P(m)$, we have
\begin{align}
    \label{assoc1}
    (a\circ_i b)\circ_j c=a\circ_i(b\circ_{j-i+1}c)
\end{align}
for all $1\leq i\leq k$, $i\leq j\leq i+l-1$, and
\begin{align}
    \label{assoc2}
    (a\circ_i b)\circ_{j + l - 1} c= (a\circ_j c)\circ_{i} b
\end{align}
for all $1\leq i<j\leq k$. Note that in the special case of $k=l=m=1$, condition \eqref{assoc1} translates to ordinary associativity of a binary operation $(-\circ_1-)$, making $\P(1)$ a semigroup, whilst condition \eqref{assoc2}
does not apply.

These conditions are modeled upon the respective properties of the operation of \textit{planar rooted trees grafting} $(-\circ_i-)$. The operation amounts to attaching a planar rooted tree $b$ onto the $i$-th leaf of a tree $a$ by its root and then  enumerating all the leaves of the resulting tree consecutively by going over them from left to right. A canonical leaf enumeration and such a traversal are well-defined due the planarity assumption.
    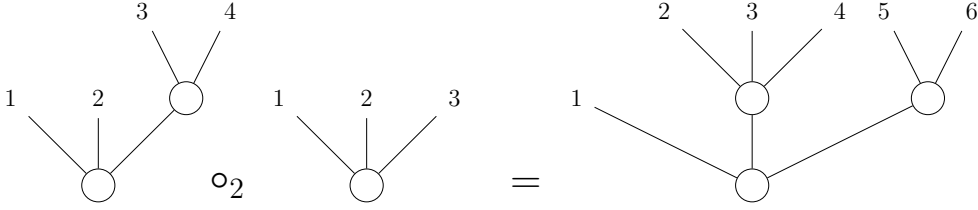
\begin{figure}[H]
        \centering
        \resizebox{0.8\textwidth}{!}{\begin{tikzpicture}[scale=1, grow=up, level distance=1.5cm, sibling distance=1.5cm, inner sep=2mm]

    \node[circle, draw] (A1) {}
        child { node[circle, draw] (A2) {}
            child { node (A3) {\Large 4} }
            child { node (A4) {\Large 3} }
        }
        child { node(A5) {\Large 2} }
        child { node (A6) {\Large 1} };

    \node at (2.2, 0) {\Huge $\circ_2$};
    \node[circle, draw, right=4cm of A1] (B1) {}
        child { node (B2) {\Large 3} }
        child { node (B3) {\Large 2} }
        child { node (B4) {\Large 1} };

    
    \node at (7.3, 0) {\Huge =};
    
    \node[circle, draw, right=6cm of B1] (C1) {}
        [sibling distance=3cm]
        child { node[draw, circle] (C6) {}
            [sibling distance=1.5cm]
            child { node (C7) {\Large 6} }
            child { node (C8) {\Large 5} }
        }
        child { node[circle, draw] (C2) {}
            [sibling distance=1.5cm]
            child { node (C3) {\Large 4} }
            child { node (C4) {\Large 3} }
            child { node (C5) {\Large 2} }
        }
        child { node(C9) {\Large 1} };

    
   \end{tikzpicture}}
        \caption{Planar rooted tree grafting}
    \end{figure}
To better understand the origin of the identities \eqref{assoc1} and \eqref{assoc2}, the reader is encouraged to draw and compare the grafting diagrams that would appear on each side of \eqref{assoc1} and \eqref{assoc2}, upon picking some planar rooted trees with $k$, $l$ and $m$ leaves for $a$, $b$ and $c$ respectively. Planar rooted trees with the grafting operations $(-\circ_i-)$ are to be regarded as \emph{the} prototypical non-$\Sigma$ operad. Any other non-$\Sigma$ operad is modeled upon it in a certain very precise sense \cite[Section 5.9.5]{LodayVallette}.
The term \emph{non-$\Sigma$} is meant to indicate the absence of any particular choice of a group action on the components $\P(n)$ of an operad $\P$, and henceforth, the absence of any equivariance conditions imposed on the partial compositions $(-\circ_i-)$. 

Some basic concepts can be readily extrapolated from semigroups and monoids to the realm of operads. Namely, given a non-$\Sigma$ operad $\P$, a subset ${\Q=\bigsqcup\limits_{n\geq 1}\Q(n)}\subseteq \P$ is said to be
\begin{itemize}
    \item a \emph{suboperad} of $\P$, if $a\circ_i b\in \Q$ for any $a\in \Q(m)$, $b\in \Q(n)$, $1\leq i\leq m$;
    \item a \emph{left ideal} of $\P$, if $a\circ_i b\in \Q$ for any $a\in \P(m)$, $b\in \Q(n)$, $1\leq i\leq m$;
    \item a \emph{right ideal} of $\P$, if $a\circ_i b\in \Q$ for any $a\in \Q(m)$, $b\in \P(n)$, $1\leq i\leq m$.
\end{itemize}
The term \emph{ideal} will refer to a subset of $\P$ that is both a left and a right ideal of $\P$. 
A set $\M=\bigsqcup\limits_{n\geq 1}\M(n)$ is said to be a \emph{left $\P$-module} 
if it is endowed with a family of maps ${(-\circ_i-):\P(m)\times \M(n) \to \M(m+n-1)}$
defined for all $m,n\geq 1$, $1\leq i\leq m$ that satisfy
\begin{align*}
    (a\circ_i b)\circ_j t=a\circ_i(b\circ_{j-i+1}t)
\end{align*}
for all $a\in \P(k)$, $b\in \P(l)$, $t\in \M(m)$ and $1\leq i\leq k$, $i\leq j\leq i+l-1$.

Similarly, a set $\M=\bigsqcup\limits_{n\geq 1}\M(n)$ is said to be a \emph{right $\P$-module} 
if it is endowed with a family of maps ${(-\circ_i-):\M(m)\times \P(n) \to \M(m+n-1)}$
defined for all $m,n\geq 1$, $1\leq i\leq m$ such that
\begin{align*}
    (t\circ_i a)\circ_j b=t\circ_i(a\circ_{j-i+1}b)
\end{align*}
for all $a\in \P(l)$, $b\in \P(m)$, $t\in \M(k)$, $1\leq i\leq k$, $i\leq j\leq i+l-1$,
and
\begin{align*}
    (t\circ_i a)\circ_{j + l - 1} b= (t\circ_j b)\circ_{i} a
\end{align*}
for all $1\leq i<j\leq k$.
Note an asymmetry in the definitions of a left and a right $\P$-module caused by an additional condition present in the latter case. This asymmetry does not emerge at the level of semigroups and monoids.

A set $\M=\bigsqcup\limits_{n\geq 1}\M(n)$ is said to be a \emph{$\P$-bimodule} if it is both a left and a right $\P$-module and, in addition, satisfies the compatibility conditions
\begin{align*}
(a\circ_i t)\circ_j b=a\circ_i(t\circ_{j-i+1}b)\\
(a\circ_i t)\circ_{j + l - 1} b= (a\circ_j b)\circ_{i} t
\end{align*}
that are to hold for any $a\in \P(k)$, $t\in \M(l)$, $b\in \M(m)$ and all appropriate values of $i$ and $j$ as in \eqref{assoc1}, \eqref{assoc2}.
\subsection{The operad of permutations.}
\label{premopsec}
For $n\geq 1$, let $Perm(n):=\S_n$ as a set. For all $m,n\geq 1$ and $1\leq i\leq m$, the \textit{partial composition}
$
{(-\circ_i-): Perm(m)\times Perm(n)\to Perm(m+n-1)}$ is defined 
as follows. In terms of the one-line notation for permutations, we set
\begin{align}
\label{permop}
a_1a_2\dots a_m \circ_i b_1 b_2\dots b_n:=
a_1'a_2'\dots {a_{i-1}'}
{(b_1+a_i-1)\,(b_2+a_i-1)\,\dots (b_n+a_i-1)} a_{i+1}'\dots a_m',
\end{align}
where $a_k'=a_k+n-1$ if $a_k>a_i$ and $a_k'=a_k$ otherwise.
For example, $31425\circ_3 231=31\underbracket[1pt]{564}27$.
A direct calculation shows that the associativity conditions \eqref{assoc1} and \eqref{assoc2} do indeed hold in $Perm$.
The disjoint union $Perm=\bigsqcup\limits_{n\geq 1}Perm(n)$ taken together with all the partial compositions $(-\circ_i-)$ is called the \emph{non-$\Sigma$ operad of permutations}. 
\begin{remark}
For any two permutations $\sigma,\tau$, the partial composition $\sigma\circ_i\tau$ is a particular case of the \textit{permutation inflation} operation \cite{AlbertAtkinsonKlazar}.
\end{remark}

In terms of bipartite matching diagrams of permutations,
a partial composition $a \circ_i b$ amounts to substituting the entire diagram $b$ for the $i$-th string of $a$. 
\begin{figure}[H]
\resizebox{\textwidth}{!}{\begin{tikzpicture}

\begin{scope}
    \foreach \i in {1,...,5} {
        \node[below] at (\i, -1) {\i};   
        \fill (\i,-1) circle (2pt);      
    }

    \foreach \i/\j in {1/3, 2/1, 3/4, 4/2, 5/5} {
        \node[above] at (\i, 1) {};  
        \fill (\i,1) circle (2pt);     
    }

    \draw (1,-1) -- (3,1);  
    \draw (2,-1) -- (1,1);  
    \draw[red] (3,-1) -- (4,1);  
    \draw (4,-1) -- (2,1);  
    \draw (5,-1) -- (5,1);  
\end{scope}

\node at (6, 0) {\Large $\circ_3$};

\begin{scope}[xshift=6cm]
    \foreach \i in {1,...,3} {
        \node[below] at (\i, -1) {\i};   
        \fill (\i,-1) circle (2pt);      
    }

    \foreach \i/\j in {1/2, 2/3, 3/1} {
        \node[above] at (\i, 1) {};  
        \fill (\i,1) circle (2pt);     
    }

    \draw (1,-1) -- (2,1);  
    \draw (2,-1) -- (3,1);  
    \draw (3,-1) -- (1,1);  
\end{scope}

\node at (10, 0) {\Large $=$};

\begin{scope}[xshift=10cm]
    \foreach \i in {1,...,7} {
        \node[below] at (\i, -1) {\i};   
        \fill (\i,-1) circle (2pt);      
    }

    \foreach \i/\j in {1/3, 2/1, 6/2, 7/7} {
        \node[above] at (\i, 1) {};  
        \fill (\i,1) circle (2pt);     
        \draw (\i, -1) -- (\j,1);
    }

    \foreach \i/\j in {3/5, 4/6, 5/4} {
        \node[above] at (\i, 1) {};  
        \fill (\i,1) circle (2pt);     
        \draw[red] (\i, -1) -- (\j,1);
    }

\end{scope}

\end{tikzpicture}}
\caption{Evaluating $31425\circ_3 231 = 3156427$ in $Perm$.}
\end{figure}
The group structure of the $\S_n$'s will not play any significant role for us. Instead, we would like to think of elements of $Perm(n)$ combinatorially, as of $n$-by-$n$ grid configurations subject to a certain constraint: there must be exactly one red cell in each row and each column of the grid. In terms of grid configurations, a partial composition $a\circ_j b$ amounts to substituting configuration $b$ for the unique red cell in the $j$-th column of $a$. Unless $b$ is the trivial $1$-by-$1$ permutation matrix, all the blue cells in the $j$-th column and the $a_j$-th row of $a$ would have to be additionally fine-grained (renormalized) by a factor of $|b|$ in order to maintain the square grid structure.
\begin{figure}[H]
\resizebox{0.6\textwidth}{!}{\input{permcompmat}}
\caption{Evaluating $31425\circ_3 231 = 3156427$ in $Perm$.}
\end{figure}
\subsection{Pattern containment and avoidance in $Perm$.}
Recall that a permutation $\sigma=\sigma_1\dots\sigma_n$ is said to \emph{contain a permutation} ${\tau=\tau_1\dots\tau_m}$ as a \emph{pattern}, denoted $\tau \leq \sigma$, if there exist $j_1<\dots<j_m$ such that the entries $\sigma_{j_1}\dots\sigma_{j_m}$ of $\sigma$ are in the same relative order as $\tau_1\dots\tau_m$.
We say that a permutation $\sigma$ \textit{avoids} $\tau$ if $\tau$ is not contained in $\sigma$ as a pattern. 
For example, $231\leq \underline{3}1\underline{42}$, but $3142$ avoids $321$.
Pattern containment is a partial order on $Perm$.
\begin{lemma}
For any $a\in Perm(m)$, $b\in Perm(n)$, $1\leq i\leq m$, we have $a\leq a\circ_i b$ and $b\leq a\circ_i b$.
\end{lemma}
\begin{proof}
The second of the two comparisons follows directly from the definition \eqref{permop}, since the\linebreak entries
${(b_1+a_i-1)\,(b_2+a_i-1)\,\dots (b_n+a_i-1)}$ are readily seen to be in the same relative order as $b_1b_2\dots b_n$.

To handle the first one, note that since $b_k=1$ for some $1\leq k\leq n$, then \eqref{percsec} contains 
a subsequence $a_1'a_2'\dots a_i'\dots a_m'$, where $a_i'=a_i$ in accordance with the notation introduced before. We claim that the entries of this subsequence are in the same relative order as $a_1\dots a_m$. That is, for any $1\leq p,q\leq m$, $a_p<a_q$ implies $a_p'<a_q'$. Indeed, if $a_i<a_p<a_q$, then $a_p'=a_p+n-1<a_q+n-1=a_q'$.
Next, if $a_p\leq a_i<a_q$, then $a_p'=a_p<a_q+n-1=a_q'$. Finally, if $a_p<a_q\leq a_i$, then $a_p'=a_p<a_q=a_q'$.
\end{proof}
\begin{corollary}
\label{avcor}
Let $\tau$ be a permutation. The set $\overline{Av}(\tau)$ of all permutations containing $\tau$ is an ideal in $Perm$. Indeed, let $b\in \overline{Av}(\tau)$. Then $\tau \leq b\leq a\circ_i b$ for any $a\in Perm(m)$. Thus, $\overline{Av}(\tau)$ is a left ideal. Similarly, $\tau\leq b\leq b\circ_j a$, showing that it is a right ideal as well.
\end{corollary}
On the contrary, pattern-avoiding families $Av(\tau_1,\dots,\tau_k):=\{\sigma\in Perm|\, \tau_i\not{\leq}\sigma\text{ for all }1\leq i\leq k\}$ tend not to be closed under the partial compositions in $Perm$. As an example, $231$ avoids the pattern $321$, but $231\circ_1 231$ does not. Indeed, $321\leq 34251=231\circ_1 231$. In this regard, pattern-avoiding families that form suboperads of $Perm$ are somewhat exceptional and are related to avoidance of \emph{simple} permutations, as we discuss in section \ref{permpres}.

\subsection{The associative operad $As$.}
\label{AsOpEx}
We consider a particularly simple example of a non-$\Sigma$ operad, called \emph{the associative operad}. Specifically, we set $As(n):=\{pt_n\}$, a one-point set for all $n\geq 2$. Here, we consider the \emph{non-unital} version of the associative operad, where $As(1):=\varnothing$. 
The partial compositions are defined by setting $pt_m\circ_i pt_n:=pt_{m+n-1}$ for all $m,n\geq 2$, $1\leq i\leq m$. One readily verifies that \eqref{assoc1} and \eqref{assoc2} do indeed hold in $As$. The operad $As$ may be viewed as a counting operad in that it merely keeps track of the number of leaves of the trees involved in a tree grafting operation. In particular, we may think of the sole element $pt_m$ of $As(m)$ as of a canonical representative of a tree with $m$ leaves. Namely, we take it to be a planar rooted corolla with $m$ leaves. A partial composition amounts to tree grafting followed by an edge contraction.
\begin{center}
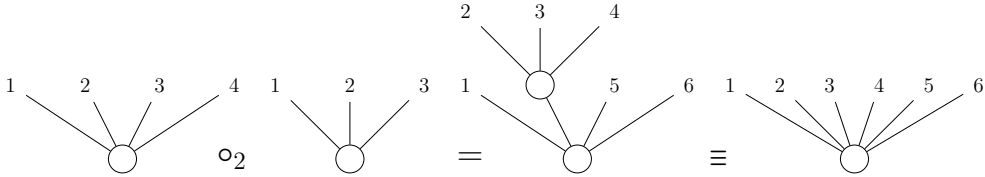

    \resizebox{0.8\textwidth}{!}{\begin{tikzpicture}[scale=1, grow=up, level distance=1.5cm, sibling distance=1.5cm, inner sep=2mm]

    \node[circle, draw] (A1) {}
        child { node(A4) {\Large 4} }
        child { node(A4) {\Large 3} }
        child { node(A3) {\Large 2} }
        child { node (A2) {\Large 1} };

    \node at (2.2, 0) {\Huge $\circ_2$};
    \node[circle, draw, right=4cm of A1] (B1) {}
        child { node (B2) {\Large 3} }
        child { node (B3) {\Large 2} }
        child { node (B4) {\Large 1} };

    
    \node at (7, 0) {\Huge =};
    
    \node[circle, draw, right=4cm of B1] (C1) {}
        child { node(C6) {\Large 6}}
        child { node(C7) {\Large 5}}
        child { node[circle, draw] (C2) {}
            [sibling distance=1.5cm]
            child { node (C3) {\Large 4} }
            child { node (C4) {\Large 3} }
            child { node (C5) {\Large 2} }        
        }
        child { node(C0) {\Large 1}};
    \node at (12, 0) {\Huge $\equiv$};
    \node[circle, draw, right=5cm of C1] (D0) {}
        [sibling distance=1cm]
        child { node(D6) {\Large 6}}
        child { node(D5) {\Large 5}}        
        child { node(D4) {\Large 4}}
        child { node(D3) {\Large 3}}
        child { node(D2) {\Large 2}}
        child { node(D1) {\Large 1}};
    
   \end{tikzpicture}}
    \captionof{figure}{Evaluating a partial composition in $As$.}
\end{center}
\subsection{The free product $As*As$.}
The free product $As*As$ of $As$ with itself can be defined analogously to the free product of monoids by virtue of a certain universal property. More explicitly, the elements of $As*As$ can be identified with the elements $pt'_m$ and $pt''_n$ of two copies of $As$ for all $m,n\geq 2$ and their formal iterated $(-\circ_i-)$-products, modulo the associativity conditions \eqref{assoc1}, \eqref{assoc2} and the partial composition relations for $As$ that hold separately for the $pt'_m$'s and $pt''_n$'s.
Graphically, the operad $As*As$ can be thought of as an operad of bicolored trees. These are planar rooted trees, where each vertex has at least two descendants and is colored either red or blue in such a way that no two adjacent vertices are of the same color. The partial compositions amount to the usual grafting followed by an edge contraction, whenever two adjacent vertices of the same color appear.
\begin{figure}[H]
    \centering
    \resizebox{0.65\textwidth}{!}{\begin{tikzpicture}[scale=1, grow=up, level distance=1.5cm, sibling distance=1.5cm, inner sep=2mm]

    \node[circle, draw, fill=red!40] (A1) {}
        child { node[circle, draw, fill=cyan!40] (A2) {}
            child { node (A3) {\Large 4} }
            child { node (A4) {\Large 3} }
        }
        child { node(A5) {\Large 2} }
        child { node (A6) {\Large 1} };

    \node at (2.2, 0) {\Huge $\circ_2$};
    \node[circle, draw, right=4cm of A1, fill=cyan!40] (B1) {}
        child { node (B2) {\Large 3} }
        child { node (B3) {\Large 2} }
        child { node (B4) {\Large 1} };
    
    \node at (7.3, 0) {\Huge =};
    
    \node[circle, draw, right=6cm of B1, fill=red!40] (C1) {}
        [sibling distance=3cm]
        child { node[draw, circle, fill=cyan!40] (C6) {}
            [sibling distance=1.5cm]
            child { node (C7) {\Large 6} }
            child { node (C8) {\Large 5} }
        }
        child { node[circle, draw, fill=cyan!40] (C2) {}
            [sibling distance=1.5cm]
            child { node (C3) {\Large 4} }
            child { node (C4) {\Large 3} }
            child { node (C5) {\Large 2} }
        }
        child { node(C9) {\Large 1} };

    
   \end{tikzpicture}}\\
    \resizebox{0.65\textwidth}{!}{\begin{tikzpicture}[scale=1, grow=up, level distance=1.5cm, sibling distance=1.5cm, inner sep=2mm]

    \node[circle, draw, fill=cyan!40] (A1) {}
        child { node[circle, draw, fill=red!40] (A2) {}
            child { node (A3) {\Large 4} }
            child { node (A4) {\Large 3} }
        }
        child { node(A5) {\Large 2} }
        child { node (A6) {\Large 1} };

    \node at (2.2, 0) {\Huge $\circ_2$};
    \node[circle, draw, right=4cm of A1, fill=cyan!40] (B1) {}
        child { node (B2) {\Large 3} }
        child { node (B3) {\Large 2} }
        child { node (B4) {\Large 1} };

    
    \node at (7.3, 0) {\Huge =};
    
    \node[circle, draw, right=6cm of B1, fill=cyan!40] (C1) {}
        [sibling distance=1.5cm]
        child { node[draw, circle, fill=red!40] (C6) {}
            [sibling distance=1.5cm]
            child { node (C7) {\Large 6} }
            child { node (C8) {\Large 5} }
        }
        child { node (C3) {\Large 4} }
            child { node (C4) {\Large 3} }
            child { node (C5) {\Large 2} }
        child { node(C9) {\Large 1} };

   \end{tikzpicture}}
    \caption{Evaluating partial compositions of planar rooted bicolored trees}
\end{figure}
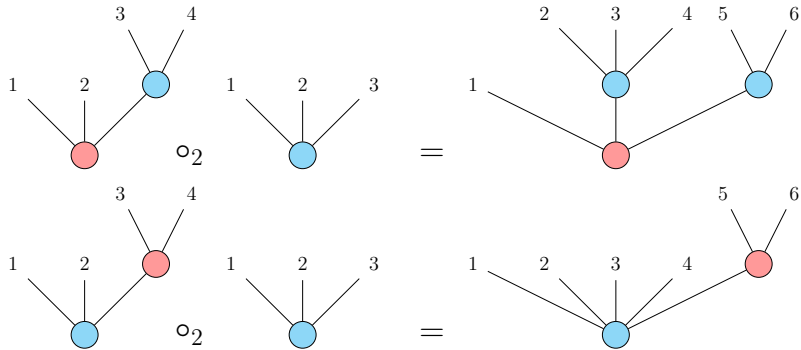
Planar rooted bicolored trees is another family of objects that are known to be counted by the big Schr\"oder numbers. That is, these are the cardinalities of the components of $As*As$.
\section{The suboperad of separable permutations}
\label{percsec}
Let $Perm_2$ be the suboperad of $Perm$ generated by the elements $12$ and $21$ in arity $2$. This is the suboperad of \textit{separable permutations}. Indeed, the direct and the skew sum of two permutations $\sigma$ and $\tau$ can be produced as $\sigma\oplus \tau=(12\circ_1 \sigma)\circ_{|\sigma|+1} \tau$ and 
$\sigma\ominus \tau=(21\circ_1 \sigma)\circ_{|\sigma|+1} \tau$ respectively. The inductive definition of separable permutations ensues.
In terms of pattern avoidance, separable permutations can be characterized as those that avoid the patterns 2413 and 3142 \cite[Section 2.2.5]{Kitaev}. 
\begin{proposition}
Operads $Perm_2$ and $As*As$ are isomorphic.
\end{proposition}
\begin{proof}
A bijection that respects the partial compositions maps a $n$-leaf corolla with a red vertex to the principal diagonal configuration of size $n$, and a $n$-leaf corolla with a blue vertex to its mirror image.
\begin{table}[H]
\begin{tabular}{cc}
 \resizebox{4cm}{!}{\begin{tikzpicture}[scale=1, grow=up, level distance=1.5cm, sibling distance=1.5cm, inner sep=2mm]

\node[circle, draw, right=4cm of A1, fill=red!40] (B1) {}
child { node (B2) {\Large 4} }
child { node (B3) {\Large 3} }
child { node (B4) {\Large 2} }
child { node (B5) {\Large 1} };

\end{tikzpicture}} & \resizebox{2.5cm}{!}{\includegraphics{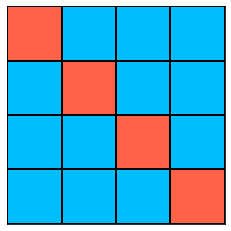}}\\
 \resizebox{4cm}{!}{\begin{tikzpicture}[scale=1, grow=up, level distance=1.5cm, sibling distance=1.5cm, inner sep=2mm]

\node[circle, draw, right=4cm of A1, fill=cyan!40] (B1) {}
child { node (B2) {\Large 4} }
child { node (B3) {\Large 3} }
child { node (B4) {\Large 2} }
child { node (B5) {\Large 1} };

\end{tikzpicture}} & \resizebox{2.5cm}{!}{\includegraphics{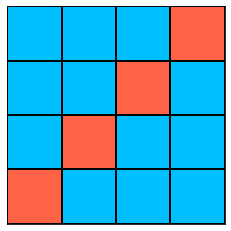}} \\
 \resizebox{5cm}{!}{\begin{tikzpicture}[scale=1, grow=up, level distance=1.2cm, sibling distance=1.6cm, inner sep=2mm]

\node[circle, draw, right=6cm of B1, fill=red!40] (C1) {}
[sibling distance=3cm]
child { node[draw, circle, fill=cyan!40] (C6) {}
    [sibling distance=1.5cm]
           child { node[circle, draw, fill=red!40] (C2) {}
    [sibling distance=1.5cm]
    child { node (C4) {\Large 7} }
    child { node (C5) {\Large 6} }
}
    child { node (C8) {\Large 5} }
}
child { node[circle, draw, fill=cyan!40] (C2) {}
    [sibling distance=1.5cm]
    child { node (C3) {\Large 4} }
    child { node (C4) {\Large 3} }
    child { node (C5) {\Large 2} }
}
child { node(C9) {\Large 1} };
    
\end{tikzpicture}} & \resizebox{2.5cm}{!}{\includegraphics{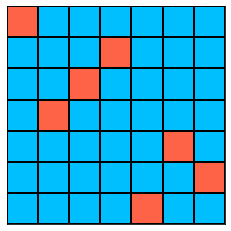}}
\end{tabular}
\end{table}
\end{proof}
\noindent
\begin{remark}
There is another instance of operad $As*As$ emerging in the context of the Ising model. 
Namely, bicolored trees can be used to encode the topological types of the nested insulated spin domain configurations on the lattice in the ordered phase. 
The correspondence is established by associating a colored $n$-leaf corolla to a simply-connected domain of the same respective color containing $n$ subdomains of the opposite color labeled by $1,\dots,n$. 
\begin{table}[H]
\centering
\begin{tabular}{cc}
 \resizebox{4cm}{!}{\begin{tikzpicture}[scale=1, grow=up, level distance=1.5cm, sibling distance=1.5cm, inner sep=2mm]

\node[circle, draw, right=4cm of A1, fill=red!40] (B1) {}
child { node (B2) {\Large 4} }
child { node (B3) {\Large 3} }
child { node (B4) {\Large 2} }
child { node (B5) {\Large 1} };

\end{tikzpicture}} & \resizebox{2.6cm}{!}{\includegraphics{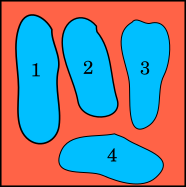}}\\
 \resizebox{4cm}{!}{\begin{tikzpicture}[scale=1, grow=up, level distance=1.5cm, sibling distance=1.5cm, inner sep=2mm]

\node[circle, draw, right=4cm of A1, fill=cyan!40] (B1) {}
child { node (B2) {\Large 4} }
child { node (B3) {\Large 3} }
child { node (B4) {\Large 2} }
child { node (B5) {\Large 1} };

\end{tikzpicture}} & \resizebox{2.6cm}{!}{\includegraphics{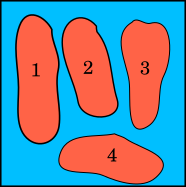}} \\
 \resizebox{5cm}{!}{\begin{tikzpicture}[scale=1, grow=up, level distance=1.2cm, sibling distance=1.6cm, inner sep=2mm]

\node[circle, draw, right=6cm of B1, fill=red!40] (C1) {}
[sibling distance=3cm]
child { node[draw, circle, fill=cyan!40] (C6) {}
    [sibling distance=1.5cm]
           child { node[circle, draw, fill=red!40] (C2) {}
    [sibling distance=1.5cm]
    child { node (C4) {\Large 7} }
    child { node (C5) {\Large 6} }
}
    child { node (C8) {\Large 5} }
}
child { node[circle, draw, fill=cyan!40] (C2) {}
    [sibling distance=1.5cm]
    child { node (C3) {\Large 4} }
    child { node (C4) {\Large 3} }
    child { node (C5) {\Large 2} }
}
child { node(C9) {\Large 1} };
    
\end{tikzpicture}} & \resizebox{2.6cm}{!}{\includegraphics{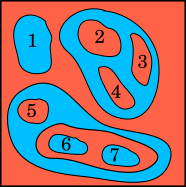}}
\end{tabular}
\end{table}
In terms of this correspondence, the partial composition $a\circ_i b$ amounts to substituting the domain distribution pattern $b$ for the $i$-th domain of $a$. A reader familiar with the notion of the little $2$-disk operad $\mathcal{D}_2$ may think of this as of a non-$\Sigma$ suboperad of the desymmetrization of $\pi_0(\mathcal{D}_2*\mathcal{D}_2)$.
\end{remark}
Another combinatorial model of operad $As*As$, and thus of separable permutations, can be constructed in terms of \emph{plane series-parallel} networks. To define these, we first recall the notion of a \emph{two-terminal} graph.
 Such a graph is a finite connected directed graph, possibly with multiple edges, but without loops, with two distinct vertices -- the \emph{source} and the \emph{sink}, marked ``$+$" and ``$-$" respectively. 
 In addition, we assume that all the edges of the graph are labeled uniquely by $1,\dots,m$. All three pieces of data decorating such a graph -- edge orientations, the choice of the source and the sink in the graph, and the edge labels, are mutually independent. In particular, the set $\mathcal{G}_2(m)$ of all two-terminal graphs with $m$ edges is acted on faithfully by $(\Z_2)^m$,  $\S_m$ and  $\Z_2$ accounting for edge orientation flips, edge relabelings and the source-sink swap respectively.

Let $a$ and $b$ be two-terminal graphs with $m$ and $n$ edges respectively. 
For $1\leq i\leq m$, the partial composition $a\circ_i b$ is defined by replacing the $i$-th edge $u\to v$ of $a$ by the entire graph $b$, with the source and the sink of $b$ mapped to $u$ and $v$ respectively. The resulting graph has $m+n-1$ oriented edges that we re-index in our usual way; cf. example \ref{AsOpEx}. The source and sink of $a$ retain their respective properties in $a\circ_i b$.
\begin{figure}[H]
    \centering
    \resizebox{0.9\textwidth}{!}{\begin{tikzpicture}[inner sep=1mm]

    \begin{scope}
    \node[draw,circle,fill=green!40,minimum size=6mm] (A) at (0,0) {\large{$+$}};
    \node[draw,circle,fill=green!40,minimum size=6mm] (B) at (2,0) {};
    \node[draw,circle,fill=green!40,minimum size=6mm] (C) at (1,1.5) {};
    \node[draw,circle,fill=green!40,minimum size=6mm] (D) at (3,1.5) {\large{$-$}};
    
    \draw[->, bend right=20, >=Latex, line width=0.5mm] (A) to node[sloped, below] {1} (B);
    \draw[->, bend left=20, >=Latex, line width=0.5mm] (B) to node[sloped, above] {2} (C);
    \draw[->, bend right=30, >=Latex, line width=0.5mm] (D) to node[sloped, above] {3} (C);
    \draw[->, bend left=30, >=Latex, line width=0.5mm] (A) to node[sloped, above] {4} (C);
    \draw[->, bend right=30, >=Latex, line width=0.5mm] (B) to node[sloped, below] {5} (D);
    \end{scope}

\node at (4.8, 1) {\huge $\circ_3$};
    
   \begin{scope}[shift={(6,1)}]
    \node[draw,circle,fill=green!40,minimum size=6mm] (A) at (0,0) {\large{$+$}};
    \node[draw,circle,fill=green!40,minimum size=6mm] (B) at (1.5,0) {};    
    \node[draw,circle,fill=green!40,minimum size=6mm] (D) at (3,0) {\large{$-$}};
    
    \draw[->, bend right=40, >=Latex, line width=0.5mm] (A) to node[sloped, below] {1} (B);
    \draw[->, bend left=40, >=Latex, line width=0.5mm] (A) to node[sloped, above] {2} (B);    
    \draw[->, bend left=20, >=Latex, line width=0.5mm] (B) to node[sloped, above] {3} (D);
    \end{scope}

    \node at (10, 1) {\huge =};
    
    \begin{scope}[shift={(11,0)}]
    \node[draw,circle,fill=green!40,minimum size=6mm] (A) at (0,0) {\large{$+$}};
    \node[draw,circle,fill=green!40,minimum size=6mm] (B) at (2,0) {};
    \node[draw,circle,fill=green!40,minimum size=6mm] (C) at (1,1.5) {};
    \node[draw,circle,fill=green!40,minimum size=6mm] (D) at (5,1) {\large{$-$}};
    \node[draw,circle,fill=green!40,minimum size=6mm] (E) at (3.5,1.5) {};
    
    \draw[->, bend right=20, >=Latex, line width=0.5mm] (A) to node[sloped, below] {1} (B);
    \draw[->, bend left=20, >=Latex, line width=0.5mm] (B) to node[sloped, above] {2} (C);
    
    \draw[->, bend right=30, >=Latex, line width=0.5mm] (D) to node[sloped, above] {4} (E);
    \draw[->, bend left=30, >=Latex, line width=0.5mm] (D) to node[sloped, above] {3} (E);
    \draw[->, bend right=20, >=Latex, line width=0.5mm] (E) to node[sloped, above] {5} (C);
    
    \draw[->, bend left=30, >=Latex, line width=0.5mm] (A) to node[sloped, above] {6} (C);
    \draw[->, bend right=30, >=Latex, line width=0.5mm] (B) to node[sloped, below] {7} (D);
    \end{scope}

   \end{tikzpicture}}
    \caption{Evaluating a partial composition of two-terminal graphs.}
\end{figure}
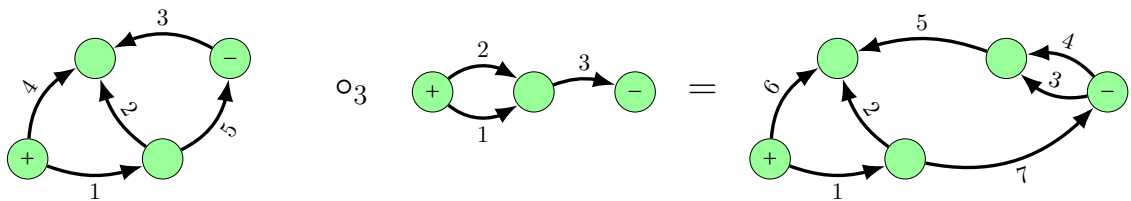
\noindent
This operation turns the collection $\mathcal{G}_2=\{\mathcal{G}_2(m)\}_{m\geq 1}$ into a non-$\Sigma$ operad.
In fact, the partial compositions $(-\circ_i-)$ defined above exhibit certain equivariant properties with respect to the group actions mentioned above, but we will not need this fact for now.

Consider \emph{plane} two-terminal graphs.
By definition, such a graph is planar and comes with a chosen embedding into a Euclidean plane determined up to an orientation-preserving isotopy of the plane. 
\begin{figure}[H]
    \centering
    \resizebox{0.4\textwidth}{!}{\begin{tikzpicture}[
    every node/.style={draw, circle, fill=green!40, minimum size=6mm, inner sep=1mm},
    every edge/.style={->, >=latex, line width=0.5mm},
    label style/.style={sloped},
    edge label/.style={fill=none, draw=none, inner sep=0mm}
]

    \begin{scope}[shift={(0,0)},scale=0.7]
        \node (A) at (0,0) {$+$};
        \node (B) at (1.5,0) {};      
        \node (C) at (3,0) {$-$};

        \draw (A) edge[draw, bend left=40] node[edge label, above] {1} (B);
        \draw (B) edge[draw, bend left=40] node[edge label, above] {2} (C);
        \draw (A) edge[draw, bend right=40] node[edge label, below] {3} (C);
    \end{scope}

    \begin{scope}[shift={(4,0)},scale=0.7]
        \node (A) at (0,0) {$+$};
        \node (B) at (1.5,0) {};      
        \node (C) at (3,0) {$-$};

        \draw (A) edge[draw, bend right=40] node[edge label, below] {1} (B);
        \draw (B) edge[draw, bend right=40] node[edge label, below] {2} (C);
        \draw (A) edge[draw, bend left=40] node[edge label, above] {3} (C);
    \end{scope}

\end{tikzpicture}}
    \caption{Two non-isomorphic plane two-terminal graphs.}
\end{figure}
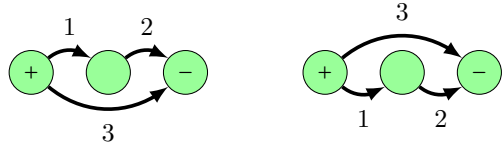
\noindent
Using the partial composition operation of two-terminal graphs introduced above, which is now assumed to preserve the plane structure, we can define the non-$\Sigma$ operad $\mathcal{SP}$ of \emph{plane series-parallel networks} by taking the plane two-terminal graphs of arity $2$ shown below as the generators.
\begin{figure}[H]
    \centering
    \resizebox{0.38\textwidth}{!}{\begin{tikzpicture}[
    every node/.style={draw, circle, fill=green!40, minimum size=6mm, inner sep=1mm},
    every edge/.style={->, >=latex, line width=0.5mm},
    label style/.style={sloped},
    edge label/.style={fill=none, draw=none, inner sep=0mm}
]

    \begin{scope}[shift={(4,0)}]
        \node (A) at (0,0) {$+$};
        \node (B) at (2,0) {$-$};    

        \draw (A) edge[draw, bend right=40] node[edge label, below] {2} (B);
        \draw (A) edge[draw, bend left=40] node[edge label, above] {1} (B);        
    \end{scope}

    \begin{scope}[shift={(0,0)}]
        \node (A) at (0,0) {$+$};
        \node (B) at (1.5,0) {};      
        \node (C) at (3,0) {$-$};

        \draw (A) edge[draw] node[edge label, above] {1} (B);
        \draw (B) edge[draw] node[edge label, above] {2} (C);
    \end{scope}

\end{tikzpicture}}
    \label{SPgraph}
\end{figure}
Any plane two-terminal graph constructed through iterative compositions of these two primitives is connected, directed, acyclic, with all paths oriented from the source to the sink, and edges equipped with a canonical enumeration. The latter, upon picking an appropriate graph embedding, can be characterized by recursively decomposing the graph in the left-to-right and top-to-bottom directions.
Note that each of the two generators gives rise to its own copy of the non-$\Sigma$ non-unital associative operad $As$ within $\mathcal{SP}$. 
\begin{table}[H]
\begin{tabular}{ccc}
 \resizebox{4.5cm}{!}{\begin{tikzpicture}[
    every node/.style={draw, circle, fill=green!40, minimum size=6mm, inner sep=1mm},
    every edge/.style={->, >=latex, line width=0.5mm},
    edge label/.style={fill=none, font=\large, draw=none},
]

\node (A) at (0,0) {$+$};
\node (B1) at (1,0) {};
\node (B2) at (2,0) {};
\node (B3) at (3,0) {};
\node (C) at (4,0) {$-$};

\draw (A) edge[draw] node[above,edge label] {1} (B1);
\draw (B1) edge[draw] node[above,edge label] {2} (B2);
\draw (B2) edge[draw] node[above,edge label] {3} (B3);
\draw (B3) edge[draw] node[above,edge label] {4} (C);

\end{tikzpicture}} &
 \resizebox{4cm}{!}{\begin{tikzpicture}[scale=1, grow=up, level distance=1.5cm, sibling distance=1.5cm, inner sep=2mm]

\node[circle, draw, right=4cm of A1, fill=cyan!40] (B1) {}
child { node (B2) {\Large 4} }
child { node (B3) {\Large 3} }
child { node (B4) {\Large 2} }
child { node (B5) {\Large 1} };

\end{tikzpicture}} \\
 \resizebox{3cm}{!}{\begin{tikzpicture}[
    every node/.style={draw, circle, fill=green!40, minimum size=4mm, inner sep=1mm},
    label style/.style={sloped, font=\large},
    edge label/.style={fill=white, font=\large, draw=none},
    every edge/.style={->, >=latex, line width=0.4mm} 
]

\node (A) at (0,0) {$+$};
\node (B) at (2,0) {$-$};    

\draw (A) edge[draw, bend left=55] node[edge label, above] {1} (B);
\draw (A) edge[draw, bend left=25] node[edge label] {2} (B);
\draw (A) edge[draw, bend right=25] node[edge label] {3} (B);
\draw (A) edge[draw, bend right=55] node[edge label, below] {4} (B);

\end{tikzpicture}} &
 \resizebox{4cm}{!}{\begin{tikzpicture}[scale=1, grow=up, level distance=1.5cm, sibling distance=1.5cm, inner sep=2mm]

\node[circle, draw, right=4cm of A1, fill=red!40] (B1) {}
child { node (B2) {\Large 4} }
child { node (B3) {\Large 3} }
child { node (B4) {\Large 2} }
child { node (B5) {\Large 1} };

\end{tikzpicture}} \\
 \resizebox{4.8cm}{!}{\begin{tikzpicture}[
    every node/.style={draw, circle, fill=green!40, minimum size=6mm, inner sep=1mm},
    label style/.style={sloped, font=\huge},
    edge label/.style={fill=none, draw=none, inner sep=0mm, font=\huge},
    every edge/.style={->, draw, >=Latex, line width=0.5mm}
]

\node[draw,circle] (A) at (0,0) {$+$};
\node[draw,circle] (B) at (2,0) {};  
\node[draw,circle] (C) at (4,0) {};
\node[draw,circle] (D) at (6,0) {$-$};
\node[draw,circle] (E) at (3,-2) {};

\draw (A) edge[bend left=50] node[edge label, above] {1} (D);
\draw (A) edge[bend left=20] node[edge label, above] {2} (B);
\draw (B) edge[bend left=20] node[edge label, above] {3} (C);
\draw (C) edge[bend left=20] node[edge label, above] {4} (D);

\draw (A) edge[bend right=20] node[edge label, below] {5} (E);
\draw (E) edge[bend right=30] node[edge label, below] {7} (D);
\draw (E) edge[bend left=30] node[edge label, below] {6} (D); 

\end{tikzpicture}} &
 \resizebox{5cm}{!}{\begin{tikzpicture}[scale=1, grow=up, level distance=1.2cm, sibling distance=1.6cm, inner sep=2mm]

\node[circle, draw, right=6cm of B1, fill=red!40] (C1) {}
[sibling distance=3cm]
child { node[draw, circle, fill=cyan!40] (C6) {}
    [sibling distance=1.5cm]
           child { node[circle, draw, fill=red!40] (C2) {}
    [sibling distance=1.5cm]
    child { node (C4) {\Large 7} }
    child { node (C5) {\Large 6} }
}
    child { node (C8) {\Large 5} }
}
child { node[circle, draw, fill=cyan!40] (C2) {}
    [sibling distance=1.5cm]
    child { node (C3) {\Large 4} }
    child { node (C4) {\Large 3} }
    child { node (C5) {\Large 2} }
}
child { node(C9) {\Large 1} };
    
\end{tikzpicture}} \\
\end{tabular}
\end{table}
\noindent
For an account of \emph{non-plane} unlabeled series-parallel networks, which are counted by MacMahon's numbers (A000084 \cite{oeisA000084}), and their relation to the symmetric commutative operad see \cite{Sartayev}.

\subsection{Combinatorics of configurations.}
Before we turn to characterization of percolating permutations matrices as those corresponding to separable permutations, we make a few preliminary observations concerning the combinatorics of final configurations.
Let $R_1$ and $R_2$ be non-overlapping rectangular regions in a grid that share a corner point.
That is, up to relabeling, one of the following holds:
\begin{enumerate}
    \item $(i,j)$ is the bottom-right corner cell of $R_1$, $(i+1,j+1)$ is the upper-left corner cell of $R_2$;
    \item 
    $(i,j)$ is the upper-right corner cell of $R_1$, $(i-1,j+1)$ is the bottom-left corner cell of $R_2$.
\end{enumerate}
In particular, in a $1$-by-$m$ or $n$-by-$1$ grid there are no rectangles that share a corner point.
\begin{lemma}
\label{union_lemma} 
Let $\sigma_1$, $\sigma_2$ be configurations such that $\sigma_1$ spans $R_1$, $\sigma_2$ spans $R_2$ and let $R$ be the smallest rectangular region in the grid containing both $R_1$ and $R_2$. If $R_1$ and $R_2$ share a corner point, then $\sigma_1\cup \sigma_2$ spans $R$.
\end{lemma}
\begin{proof}
We seek to show that $R\subset \overline{\sigma_1\cup \sigma_2}$.
Note that for any two configurations $\tau\subset \sigma$ we have $\overline{\tau}\subset \overline{\sigma}$, and $\overline{\overline{\sigma}}=\overline{\sigma}$.
Since $\overline{\sigma_1}, \overline{\sigma_2}\subset \overline{\sigma_1\cup \sigma_2}$,
then 
$R_1\cup R_2 \subset \overline{\sigma_1}\cup\overline{\sigma_2}\subset \overline{\sigma_1\cup \sigma_2}$. It remains to notice that $\overline{R_1 \cup R_2}=R$. Indeed, 
the complement of $R_1\cup R_2$ in $R$ is the union of two rectangular subregions bounded by red cells placed along a horizontal and a vertical edge. Such a configuration is percolating by an obvious inductive argument.
\begin{table}[H]
\centering
\begin{tabular}{cc}
\resizebox{!}{2.3cm}{\includegraphics{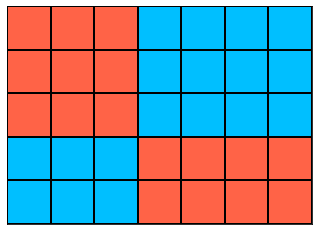}}     & 
\resizebox{!}{2.3cm}{\includegraphics{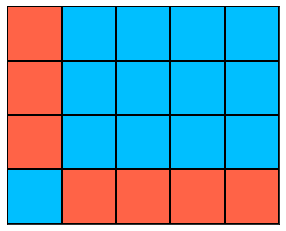}}\\
\end{tabular}
\end{table}
\end{proof}
As usual, for any two cells $(i,j)$, $(i',j')$, the Manhattan, or the $\ell_1$, distance between them is given by $$d((i,j),(i',j'))=|i-i'|+|j-j'|.$$ Let $S$ be an arbitrary subset of a $n$-by-$m$ grid with cell colors ignored. The $k$-\emph{neighborhood} $U_k(S)$ of $S$ is the set of all cells $(i',j')$ of the grid such that there exists $(i,j)\in S$ with $d((i,j),(i',j'))\leq k$. In particular, $U_0(S)=S$. The $k$-\emph{collar} $C_k(S)$ of $S$ is $U_k(S)\setminus S$.
\begin{proposition}
\label{collar_lemma}
For any initial configuration $\sigma$, the final configuration $\overline{\sigma}$ is a disjoint union of rectangular regions such that the $2$-collar of any such region $R$ does not contain any red cells.  In particular, $\sigma$ is percolating precisely when there is only one such region $R$ and $C_2(R)=\varnothing$.
\end{proposition}
\begin{proof}
It suffices to focus on the case of a non-percolating $\sigma$.
Let $R$ be a connected, in the sense of the vertical and horizontal adjacency, component of $\overline{\sigma}$. Suppose that it is not rectangular. Then $R$ contains an $L$-shaped tromino
or a $2$-by-$2$ diagonal configuration, possibly rotated, as a subconfiguration. Both cases admit a percolation step in contradiction to $\overline{\sigma}$ being a final configuration.
\begin{table}[H]
\centering
\begin{tabular}{cc}
\resizebox{1.6cm}{!}{\includegraphics{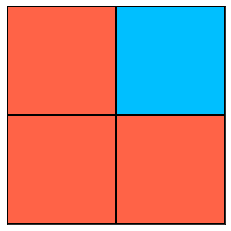}}     & 
\resizebox{1.6cm}{!}{\includegraphics{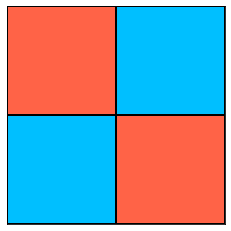}}\\
\end{tabular}
\end{table}

Now, let $R$ be a rectangular connected component of $\overline{\sigma}$.
Let $(i',j')\in C_2(R)$ be a red cell. If $(i',j')$ is at distance $1$ from some $(i,j)\in R$, then it is horizontally or vertically adjacent to $(i,j)$ and, hence, must be in $R$ by connectedness. A contradiction. 
Let $(i',j')\in C_2(R)$ be a red cell at distance $2$ from some $(i,j)\in R$. Then $(i,j)$ and $(i',j')$ either form a diagonal $2$-by-$2$ configuration or sit at the opposite ends of a horizontal or vertical segment of three cells with a blue cell in the middle. In either case, such a subconfiguration admits a percolation step, once again contradicting $\overline{\sigma}$ being a final configuration.

\begin{table}[H]
\centering
\begin{tabular}{cc}
\resizebox{4cm}{!}{\includegraphics{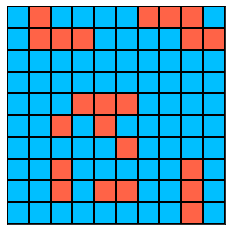}}     & 
\resizebox{4cm}{!}{\includegraphics{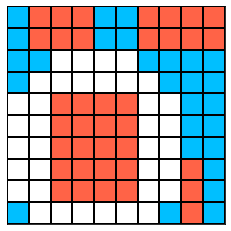}}\\
\end{tabular}
\captionof{figure}{An initial configuration $\sigma$ and the final configuration $\overline{\sigma}$. On the right figure, the $2$-collar of one of the connected components is highlighted.}
\end{table}
\end{proof}
An equivalent characterization of final configurations is that $\overline{\sigma}$ is a disjoint union of rectangular regions with pairwise disjoint $1$-neighborhoods.
The set $\mathcal{F}_{n,m}$ of all such disjoint unions in a $n$-by-$m$ grid is a \textit{convexity} structure\cite[Section 1.1]{VandeVel} on the grid. That is, $\mathcal{F}_{n,m}$, as a family of subsets of the grid, satisfies the following: (1) the empty set $\varnothing$ and the (all-red) grid are in $\mathcal{F}_{n,m}$; (2) $\mathcal{F}_{n,m}$ is closed under intersections; (3) $\mathcal{F}_{n,m}$ is closed under nested unions $\bigcup\limits_{k\geq 1}C_k$, where $C_1\subseteq C_2\subseteq \dots$. As with any convexity structure, there is the \emph{closure}, or the \emph{convex hull}, operator associated with $\mathcal{F}_{n,m}$. For a given set $\sigma$, it returns the smallest set from $\mathcal{F}_{n,m}$ containing $\sigma$. By the above proposition, this is precisely the final configuration $\overline{\sigma}$.

Let $\sigma$, $\tau$ be some configurations, where $\tau$ is in a $p$-by-$q$ grid. Given $(i,j)\in \sigma$, we perform the following operation: the cell $(i,j)$ is replaced by the entire configuration $\tau$; any cell in the $i$-th row of $\sigma$ other than $(i,j)$ is refined to a column of height $p$ of the same color; any cell in the $j$-th column of $\sigma$ other than $(i,j)$ is refined to a row of width $q$ of the same color.
The resulting configuration is denoted by $\sigma\circ_{(i,j)}\tau$.
In particular, if $\sigma$ and $\tau$ are permutation matrices, then $\sigma\circ_{(\sigma(j),j)}\tau$ is $\sigma\circ_j\tau$ in the sense of the partial compositions in $Perm$ as described in section \ref{premopsec}. A more detailed study of this operation is a subject of our forthcoming paper.
\begin{table}[H]
\centering
\begin{tabular}{cc}
\resizebox{3cm}{!}{\includegraphics{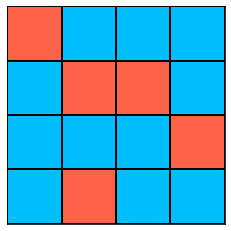}}     & 
\resizebox{3cm}{!}{\includegraphics{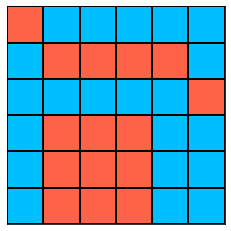}}\\
\end{tabular}
\captionof{figure}{Configurations $\sigma$ and $\sigma\circ_{(4,2)}\tau$, where $\tau$ is an all-red $3$-by-$3$ grid.}
\end{table}

\begin{lemma}
\label{insert_comp_lemma} 
Let $\sigma$ be non-percolating and $\tau$ be any configuration. For any $(i,j)\in \sigma$,
the configuration $\sigma\circ_{(i,j)}\tau$ is non-percolating.
\end{lemma}
\begin{proof}
If $\sigma\circ_{(i,j)}\tau$ is all-blue, which happens when $\sigma$ contains a single red cell $(i,j)$ and $\tau$ is all-blue,
then the statement follows immediately. 
Otherwise, let $\tau$ be such that is has at least one red cell. Since $\sigma$ is non-percolating, there exists a connected rectangular component $R$ of $\overline{\sigma}$ containing $(i,j)$ with the non-empty $2$-collar $C_2(R)$. 
Let $C'$ be the subconfiguration that $C_2(R)$ transforms to in $\sigma\circ_{(i,j)}\tau$, and $R'$ be the smallest rectangular region in $\sigma\circ_{(i,j)}\tau$ containing $R\circ_{(i,j)}\tau$. Since $C'$ consists of blue cells only, $R'$ has a non-empty $2$-colar. Then $\overline{R\circ_{(i,j)}\tau}\subset R'$ has a non-empty $2$-collar as well, and by lemma \ref{collar_lemma}, $\sigma\circ_{(i,j)}\tau$ is non-percolating.
\end{proof}
\begin{theorem}
\label{ShStThm}
All percolating $n$-by-$n$ permutation matrices for $n\geq 2$ form a suboperad $Perm_2$ of the non-$\Sigma$ operad of permutations $Perm$. The suboperad is generated by permutations $12$ and $21$ in arity $2$ and consists of all separable permutations. 
\end{theorem}
\begin{proof}
Let $\sigma$ be a percolating $n$-by-$n$ permutation matrix. We will show by induction on $n$ that $\sigma$, as a permutation, is separable. The base case is clear: both of the $2$-by-$2$ permutation matrices are percolating and the corresponding permutations are separable.
Let $n>2$. Since $\sigma$ is percolating, there exists at least one blue cell $(i,j)$ in $\sigma$ that will turn red at the very first step of the percolation process. Since the $i$-th row and the $j$-th column of $\sigma$ contain only one red cell, blue cell $(i,j)$ must appear as a blue corner of a $2$-by-$2$ diagonal configuration $\tau$. 
Without loss of generality, we may assume that $(i,j)$ is the top-left corner of $\tau$. The remaining three cases are analogous.
Let $\sigma'$ be the configuration obtained by removing the $(i+1)$-th row and the $(j+1)$-th column of $\sigma$, while contracting $\tau$ to a single red cell at $(i,j)$. That is, $\sigma'$ is a $(n-1)$-by-$(n-1)$ permutation matrix such that $\sigma=\sigma'\circ_j \tau$ for a $2$-permutation $\tau$. By lemma \ref{insert_comp_lemma}, $\sigma'$ must be percolating and, by the inductive assumption, it is separable as a permutation. Therefore, by the inductive definition of separability, $\sigma$ is separable as well.

It remains to show that any separable permutation is percolating. This is done by another round of induction on $n$, this time using lemma \ref{union_lemma}.
\end{proof}
In \cite{ShapiroStephens} the authors get the result by implicitly using the $As*As$ model of $Perm_2$ instead.

As a corollary to lemma \ref{insert_comp_lemma},
all non-percolating permutation matrices form a suboperad $NPerm_2$, which is the set-theoretic complement of $Perm_2$ in $Perm$.
In fact, by corollary \ref{avcor}, a stronger result holds, and $NPerm_2$ is an ideal in $Perm$.
As a summary to this section, all permutation matrices retain their respective percolative properties under the substitution operation, and thus, under block renormalization and reduction. Percolating permutation matrices are minimal, as percolating sets, and are closed under the partial composition within the permutation operad.

\section{A filtration of $Perm$ and generalized Schr\"oder numbers}
\label{permpres}
We consider a certain inductively defined  filtration $P_2\subset P_3\subset \dots$ of $Perm$ as of a $Perm_2$-bimodule. Namely, first we set $P_2:=Perm_2$ and $G_2:=\{12,21\}$. Then, for every $n>2$, let $G_n:=\S_n\setminus P_{n-1}(n)$, and define $P_n$ to be the $Perm_2$-bimodule generated by $G_i$ in arity $i$ for all $2\leq i\leq n$.
We can compute the following.
\begin{itemize}
\item 
$G_3=\varnothing$. All six permutations of length $3$ are representable as partial compositions of $12$ and $21$.
\item 
$G_4=\{2413,3142\}$
\begin{table}[H]
\centering
\begin{tabular}{cc}
\resizebox{2.2cm}{!}{\includegraphics{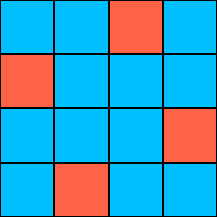}}     & 
\resizebox{2.2cm}{!}{\includegraphics{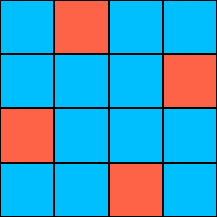}}\\
\end{tabular}
\end{table}
\item 
$G_5=\{42513,31524,35142,24153,41352,25314\}$
\begin{table}[H]
\centering
\begin{tabular}{cccc}
   \includegraphics[width=2.2cm]{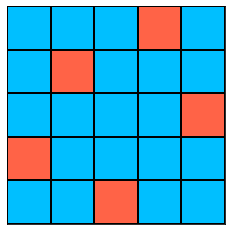}  & 
   \includegraphics[width=2.2cm]{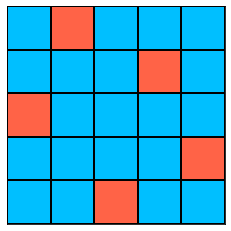}  &
   \includegraphics[width=2.2cm]{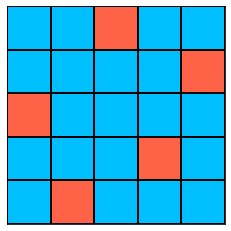}  &
   \includegraphics[width=2.2cm]{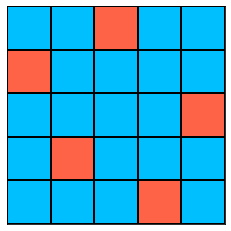} \\
     &
\includegraphics[width=2.2cm]{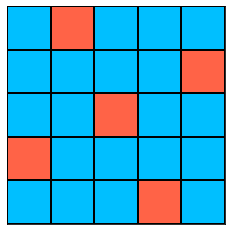}
&
\includegraphics[width=2.2cm]{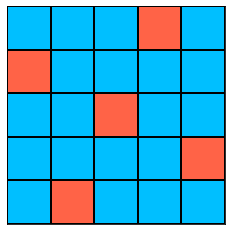}    
& 
\end{tabular}
\end{table}
\item
$G_6$ consists of $46$ permutations, $G_7$ consists of 354 permutations etc.
\begin{table}[H]
\centering
\begin{tabular}{cccc}
\resizebox{2.5cm}{!}{\includegraphics{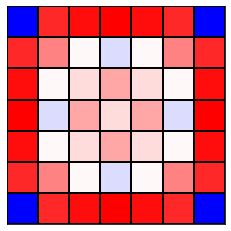}}   & 
\resizebox{2.5cm}{!}{\includegraphics{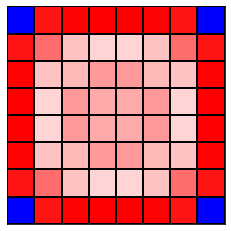}}
&
\resizebox{2.5cm}{!}{\includegraphics{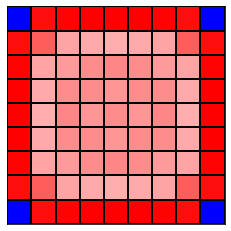}}   & 
\resizebox{2.5cm}{!}{\includegraphics{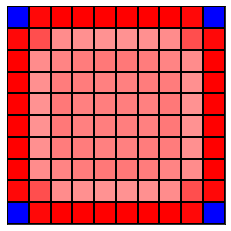}}
\end{tabular}
\captionof{figure}{The frequency maps of red cells in $G_7,G_8,G_9,G_{10}$.}
\end{table}
\end{itemize}

For $n\geq 4$, each of the generators sets $G_n$ consists of non-percolating configurations of size $n$. Hence, by lemma \ref{insert_comp_lemma}, the $Perm_2$-bimodules generated by $G_n$ consist of non-percolating sets as well.
Note that because of the overall symmetry of both the grid and the percolation rules, corresponding in terms of the underlying physical model to the symmetry of the isotropic Ising Hamiltonian, each $P_n$ splits into irreducible representations of the dihedral group of order $8$, the group of symmetries of a square. In particular, within $G_n$'s one may spot the subfamilies of non-percolating "wheels" and "sunshines" forming orbits of length $4$ and $2$ respectively.
\begin{table}[H]
\centering
\begin{tabular}{cc}
\resizebox{2.6cm}{!}{\includegraphics{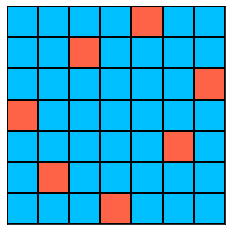}}     & 
\resizebox{2.6cm}{!}{\includegraphics{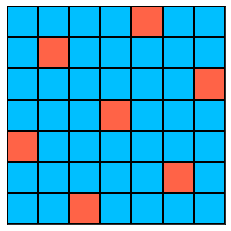}}\\
\end{tabular}
\captionof{figure}{A "wheel" and a "sunshine" in $G_7$.}
\end{table}
\subsection{A presentation of $Perm$.}
The generators $G_n$ ($n\geq 2$) are counted by the following sequence:
\begin{align}
\label{Gcardseq}
2, 0, 2, 6, 46, 354, 3106, 29926, 315862\dots
\end{align}
Up to the argument shift and the first few terms, this is sequence A078603 \cite{oeisA078603}, which is known to count the number of ways of cyclically arranging numbers $1,2,\dots,n$ so that no two adjacent entries would differ by~$1$. Combined with the observation that, by its inductive definition, the filtration $P_2\subset P_3\subset \dots \subset Perm$ is exhaustive, this leads to the following
\begin{theorem}
$Perm$ admits a presentation with generators $G_n$ indexed by oriented Hamiltonian cycles on the complement of the cyclic graph $C_{n+1}$ for $n\geq 4$, and two exceptional generators $G_2=\{12,21\}$ in arity $n=2$.
\end{theorem}
\begin{proof}
The proof amounts to constructing a model of $Perm$ in terms of oriented chord diagrams on a circle with $n+1$ cyclically enumerated marked points. Specifically, to a permutation $\sigma=\sigma_1\dots\sigma_n$, we associate an oriented circle with $n+1$ distinct points labeled by $0,1,\dots,n$ in the clockwise order. For each $1\leq i\leq n-1$, we draw an oriented chord $\sigma_i\sigma_{i+1}$ from point $\sigma_i$ to point $\sigma_{i+1}$. Furthermore, we add chords $\sigma_0 \sigma_1$ and $\sigma_n\sigma_0$ producing a cycle. 
The construction can be regarded as a ``one-point compactification" of the one-line notation for permutations.
\begin{table}[H]
\centering
\begin{tabular}{cc}
\raisebox{-.5\height}{\resizebox{!}{3cm}{\usetikzlibrary{arrows}

\begin{tikzpicture}
\newcommand{\midarrow}{\tikz 
\draw[line width=2pt, -{>[scale=3,width=5]}] (0,0) -- +(5,0);}

\draw[line width=2pt] (0,0) circle (3cm);

\draw[line width=3pt, every node/.style={sloped,allow upside down}] (90-72*0:3cm) 
-- node {\midarrow}(90-72*3:3cm) 
-- node {\midarrow} (90-72*1:3cm) 
-- node {\midarrow} (90-72*4:3cm) 
-- node {\midarrow} (90-72*2:3cm) -- 
node {\midarrow} cycle;

\foreach \i in {1,2,3,4} {
    \node[scale=1.5,draw, circle, fill=cyan!40, inner sep=3pt] at (90 - 72 * \i:3cm) {\i};
}

\node[scale=1.5,draw, circle, fill=green!40, inner sep=3pt] at (90 - 72 * 0:3cm) {0};

\end{tikzpicture}}} & 
\raisebox{-.5\height}{\resizebox{3cm}{!}{\usetikzlibrary{arrows}

\begin{tikzpicture}
\newcommand{\midarrow}{\tikz 
\draw[line width=2pt, -{>[scale=3,width=4]}] (0,0) -- +(4.6,0);}

\draw[line width=2pt] (0,0) circle (3cm);

\draw[line width=3pt, every node/.style={sloped,allow upside down}] (90-60*0:3cm) 
-- node {\midarrow}(90-60*4:3cm) 
-- node {\midarrow} (90-60*2:3cm) 
-- node {\midarrow} (90-60*5:3cm) 
-- node {\midarrow} (90-60*1:3cm) 
-- node {\midarrow} (90-60*3:3cm) -- 
node {\midarrow} cycle;

\foreach \i in {1,2,3,4,5} {
    \node[scale=1.5,draw, circle, fill=cyan!40, inner sep=3pt] at (90 - 60 * \i:3cm) {\i};
}

\node[scale=1.5,draw, circle, fill=green!40, inner sep=3pt] at (90 - 60 * 0:3cm) {0};

\end{tikzpicture}}}\\
\end{tabular}

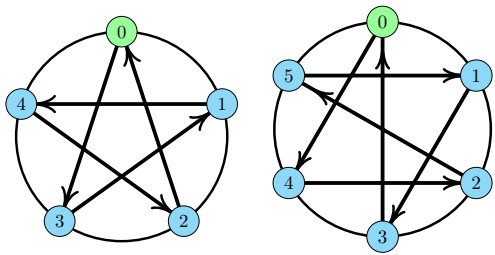
\captionof{figure}{The chord diagrams for permutations $3142\in G_4$ and $42513\in G_5$.}
\end{table}
The partial composition $\sigma\circ_i\tau$ amounts to 
gluing the marked circle $\tau$ by its $0$-th point to the $\sigma_i$-th point of $\sigma$, followed by merging two chord diagrams upon resolving a ``nodal singularity". All the marked points of a new chord diagram are to be indexed according to our usual rule, cf. example \ref{AsOpEx}.
\begin{figure}[H]
\includegraphics[width=0.57\textwidth]{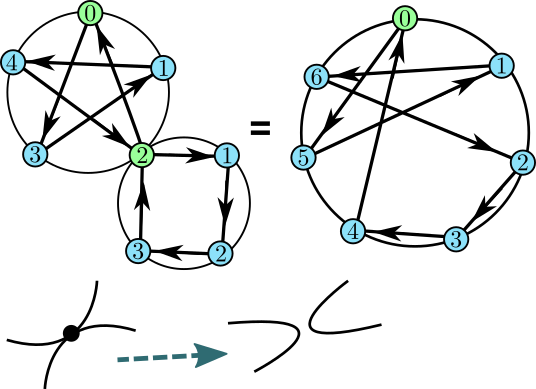}  
\caption{Evaluating $3142\circ_4 123=516234$ in $Perm$.}
\end{figure}
\end{proof}
\begin{example}
As a simple application of the above model, we may construct a family $E(m)$ of non-percolating sets of size $2m$ in a $2m$-by-$2m$ grid for $m\geq 2$. This amounts to observing that a Hamiltonian cycle of length $2m+1$ with the desired property can be constructed by going clockwise along a circle with $2m+1$ labeled points with the step of $2$.
\begin{table}[H]
\centering
\begin{tabular}{cc}
 \resizebox{4cm}{!}{\usetikzlibrary{arrows}

\begin{tikzpicture}
\newcommand{\midarrow}{\tikz 
\draw[line width=2pt, -{>[scale=3,width=4]}] (0,0) -- +(3.4,0);}

\draw[line width=2pt] (0,0) circle (3cm);

\draw[line width=3pt, every node/.style={sloped,allow upside down}] (90-60*0:3cm) 
-- node {\midarrow}(90-40*2:3cm) 
-- node {\midarrow} (90-40*4:3cm) 
-- node {\midarrow} (90-40*6:3cm) 
-- node {\midarrow} (90-40*8:3cm) 
-- node {\midarrow}(90-40*1:3cm) 
-- node {\midarrow} (90-40*3:3cm) 
-- node {\midarrow} (90-40*5:3cm) 
-- node {\midarrow} (90-40*7:3cm) 
-- node {\midarrow} cycle;

\foreach \i in {1,2,3,4,5,6,7,8} {
    \node[scale=1.5,draw, circle, fill=cyan!40, inner sep=3pt] at (90 - 40 * \i:3cm) {\i};
}

\node[scale=1.5,draw, circle, fill=green!40, inner sep=3pt] at (90 - 40 * 0:3cm) {0};

\end{tikzpicture}} & \resizebox{3.5cm}{!}{\includegraphics{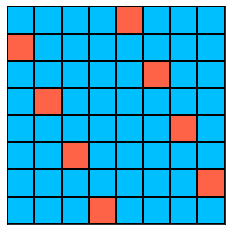}}
\end{tabular}
\captionof{figure}{$E(4)$}
\end{table}
As discussed earlier, (non-)percolating families can be naturally multiplied by means of the grid symmetries. In our case, the grid rotation by $\frac{\pi}{2}$ yields another family that we denote $E'(m)$. The corresponding Hamiltonian cycle starts at $0$ and proceeds with the constant uniform step of $m$.
\begin{table}[H]
\centering
\begin{tabular}{cc}
 \resizebox{4cm}{!}{\usetikzlibrary{arrows}

\begin{tikzpicture}
\newcommand{\midarrow}{\tikz 
\draw[line width=2pt, -{>[scale=3,width=4]}] (0,0) -- +(3.4,0);}

\draw[line width=2pt] (0,0) circle (3cm);

\draw[line width=3pt, every node/.style={sloped,allow upside down}] (90-60*0:3cm) 
-- node {\midarrow}(90-40*4:3cm) 
-- node {\midarrow} (90-40*8:3cm) 
-- node {\midarrow} (90-40*3:3cm) 
-- node {\midarrow} (90-40*7:3cm) 
-- node {\midarrow}(90-40*2:3cm) 
-- node {\midarrow} (90-40*6:3cm) 
-- node {\midarrow} (90-40*1:3cm) 
-- node {\midarrow} (90-40*5:3cm) 
-- node {\midarrow} cycle;

\foreach \i in {1,2,3,4,5,6,7,8} {
    \node[scale=1.5,draw, circle, fill=cyan!40, inner sep=3pt] at (90 - 40 * \i:3cm) {\i};
}

\node[scale=1.5,draw, circle, fill=green!40, inner sep=3pt] at (90 - 40 * 0:3cm) {0};

\end{tikzpicture}} & \resizebox{3.5cm}{!}{\includegraphics{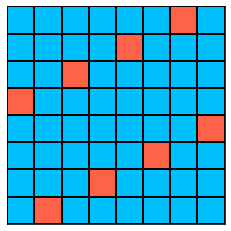}}
\end{tabular}
\captionof{figure}{$E'(4)$}
\end{table}
The reflection of a permutation matrix about the central vertical axis corresponds to flipping the Hamiltonian cycle orientation. In terminology of \cite{AlbertAtkinson}, the families of permutations obtained from $E(m)$'s by means of rotations and reflections are called the \emph{exceptional simple} permutations.
More generally, one gets a non-percolating configuration of size $n$ in a $n$-by-$n$ grid for any primitive $(n+1)$-th root of unity other than 
$e^{\pm \frac{2\pi i}{n+1}}$. The rotations and reflections of the grid can be related to transformations of the complex plane, where the roots reside.
\begin{table}[H]
\centering
\begin{tabular}{cc}
 \resizebox{4cm}{!}{\usetikzlibrary{arrows}

\begin{tikzpicture}
\newcommand{\midarrow}{\tikz 
\draw[line width=2pt, -{>[scale=3,width=4]}] (0,0) -- +(3.4,0);}

\draw[line width=2pt] (0,0) circle (3cm);


 \foreach \i [count=\j from 0] in {0, 3, 6, 9, 2, 5, 8, 1, 4, 7} {
        \coordinate (P\j) at (90-36*\i:3cm);
    }

    \draw[line width=3pt, every node/.style={sloped,allow upside down}] (P0)
    \foreach \j in {1,2,3,4,5,6,7,8,9,0} { -- node {\midarrow}(P\j) }
    -- cycle;

\foreach \i in {1,2,3,4,5,6,7,8,9} {
    \node[scale=1.5,draw, circle, fill=cyan!40, inner sep=3pt] at (90 - 36 * \i:3cm) {\i};
}

\node[scale=1.5,draw, circle, fill=green!40, inner sep=3pt] at (90 - 36 * 0:3cm) {0};

\end{tikzpicture}} & \resizebox{3.5cm}{!}{\includegraphics{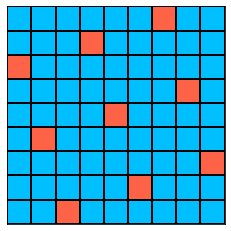}}
\end{tabular}
\captionof{figure}{A non-percolating configuration for $\zeta=e^{\frac{3\pi i}{5}}$.}
\end{table}
\noindent
For the sake of a comparison, a chord diagram of a percolating configuration corresponding to a separable permutation is shown below.
\begin{table}[H]
\centering
\begin{tabular}{cc}
 \resizebox{4cm}{!}{\begin{tikzpicture}
 \newcommand{\midarrow}{\tikz 
 \draw[line width=2pt, -{>[scale=3,width=4]}] (0,0) -- +(1.6,0);}

\draw[line width=2pt] (0,0) circle (3cm);

 \foreach \i [count=\j from 0] in {0, 3, 6, 8, 7, 9, 5, 4, 1, 2} {
        \coordinate (P\j) at (90-36*\i:3cm);
    }

    \draw[line width=3pt, every node/.style={sloped,allow upside down}] (P0)
    \foreach \j in {1,2,3,4,5,6,7,8,9,0} { -- node{\midarrow}(P\j) }
    -- cycle;

\foreach \i in {1,2,3,4,5,6,7,8,9} {
    \node[scale=1.5,draw, circle, fill=cyan!40, inner sep=3pt] at (90 - 36 * \i:3cm) {\i};
}

\node[scale=1.5,draw, circle, fill=green!40, inner sep=3pt] at (90 - 36 * 0:3cm) {0};

\end{tikzpicture}} & \resizebox{3.5cm}{!}{\includegraphics{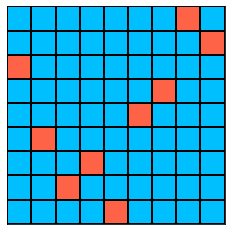}}
\end{tabular}
\captionof{figure}{The chord diagram of separable permutation $368795412\in Perm_2(9)$.}
\end{table}
\end{example}

The cardinalities of the components of the $Perm_2$-bimodules $P_n$ can be regarded as a certain generalization of the big Schr\"oder numbers. The latter measure the cardinalities of the components of $P_2=Perm_2$. We have
\begin{center}
\begin{tabular}{p{2cm}p{7cm}}
$P_2,P_3$ & $2, 6, 22, 90, 394, 1806, 8558,41586\dots$\quad \cite{oeisA006318}\\
$P_4$ & $2,6,24,114,590,3182,17522,97594\dots$ \\
$P_5$ & $2,6,24,120,674,3950,23390,138394\dots$ \\
$P_6$ & $2,6,24,120,720,4686,30842,200034\dots$ \\
$P_7$ & $2,6,24,120,720,5040,37214,270834\dots$
\end{tabular}
\end{center}
At present, the entries for $P_n$ ($n\geq 4$) do not seem to appear in the OEIS.
Note that A078603 \cite{oeisA078603} counts, up to the argument shift, the differences $n!-|P_{n-1}(n)|$.
\subsection{A dihedral operad of Hamiltonian cycles.}
The non-$\Sigma$ operad structure on oriented Hamiltonian cycles on a circle with labeled points has additional symmetries: the labels can be cyclically shifted and a Hamiltonian cycle can be reversed.
To properly encode these symmetries, one needs a certain generalization of non-$\Sigma$ operads.
Specifically, similarly to how planar rooted trees, with canonically enumerated leaves, can be equipped with the grafting operation, planar \emph{non-rooted} trees with enumerated leaves can be given a compositional structure by means of bi-indexed grafting  $(-{}_{i}\circ_j-)$. As before, let $a$, $b$ be planar trees, whose  leaves are (independently) enumerated in the clockwise traversal order. To match the notation introduce before, this time we start enumeration from $0$. On such a tree $a$ with $m$ leaves we have an operation of the cyclic shift $T$ that changes each leaf label $i$ to $i+1 (\text{mod} m)$. Another operation $S$ amounts to flipping a labeled tree and changing each leaf label $i$ to $m-i (\text{mod} m)$ so that the clockwise enumeration order is retained.

Now, $a{}_{i}\circ_j b$ is defined by gluing a tree $b$ by its $j$-th leaf onto the $i$-th leaf of a tree $a$ and enumerating the leaves of the resulting tree clockwise. If $i\neq 0$, the $0$-th leaf of $a{}_{i}\circ_j b$ is taken to be the $0$-th leaf of $a$. Otherwise, we compute $a{}_{i}\circ_j b$ as $T^{-1}(Ta{}_{1}\circ_j b)$
\begin{center}
    \resizebox{0.7\textwidth}{!}{\begin{tikzpicture}[every circle node/.style={fill=cyan!40, inner sep=3pt}]
    
    \begin{scope}[scale=0.7]    
    \node[draw, circle] (root) at (0,0.5) {};
    \node[draw, circle] (branch1) at (-1, 1.3) {};
    \node (branch2) at (1, 1.3) {2};
    \node (branch3) at (1, -0.7) {3};
    \node (branch4) at (-1, -0.7) {4};

    \draw (root) -- (branch1);
    \draw (root) -- (branch2);
    \draw (root) -- (branch3);
    \draw (root) -- (branch4);

    \node (leaf1) at (-2, 2.5) {1};
    \node(leaf5) at (-2, 1) {0};
    \draw (branch1) -- (leaf1);
    \draw (branch1) -- (leaf5);
    \end{scope}

    \begin{scope}[shift={(2.5,0.8)}, scale=0.7]    
    \node[draw, circle] (r) at (0,0) {};
    \node (a) at (-1, 1) {2};
    \node (b) at (1, 1) {0};
    \node[draw, fill=white] (c) at (0,-1.4) {1};
    \draw (r) -- (a);
    \draw (r) -- (b);
    \draw (r) -- (c);
    \end{scope}

    \node at (1.5, 0.5) {\Large ${}_{2}\circ_1$};

    \node at (3.5, 0.5) {\Large $=$};

    \begin{scope}[scale=0.7, shift={(8,-0.3)}]    
    \node[draw, circle] (root) at (0,0.5) {};
    \node[draw, circle] (branch1) at (-1, 1.3) {};
    \node[draw, circle] (branch2) at (1, 1.3) {};
    \node (b) at (0,2) {2};
    \node (c) at (2,2) {3};
    \node (branch3) at (1, -0.7) {4};
    \node (branch4) at (-1, -0.7) {5};

    \draw (root) -- (branch1);
    \draw (root) -- (branch2);
    \draw (root) -- (branch3);
    \draw (root) -- (branch4);
    \draw (b) -- (branch2);
    \draw (c) -- (branch2);

    \node (leaf1) at (-2, 2.5) {1};
    \node(leaf5) at (-2, 1) {0};
    \draw (branch1) -- (leaf1);
    \draw (branch1) -- (leaf5);
    \end{scope}
    
\end{tikzpicture}}
    
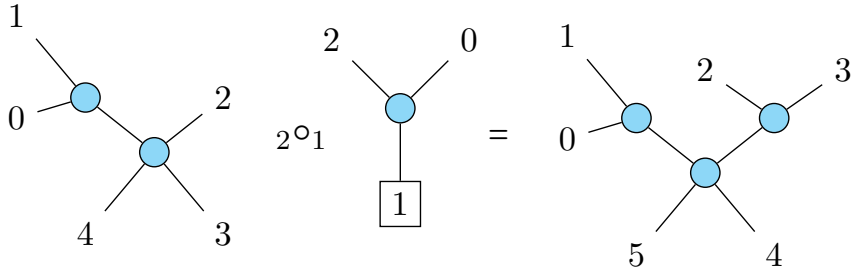
\captionof{figure}{Grafting planar non-rooted labeled trees.}
\end{center}
The operation is associative in a generalized sense, similarly to \eqref{assoc1},\eqref{assoc2}. A structure that axiomatizes the properties of this operation is that of a \textit{dihedral operad} \cite{DupontVallette}. Indeed, the set of all planar non-rooted trees with $n$ leaves labeled by $1,\dots,n$ in a cyclic order can be given the natural action of the dihedral group $\mathbb{D}ih_n$ that amounts to cyclically shifting the labels on the leaves and reversing the enumeration.
The Hamiltonian cycles model of permutations can be given the natural structure of a dihedral operad $Ham$. In this regards, permutations are a "gauge-fixed" object contained within $Ham$. Fixing the gauge amounts to choosing a designated point on a circle, the one labeled by $0$ in terminology and notation of the previous section, in each arity.
\subsection{Simple permutations.}
Starting with $n=7$, the generating sets $G_n$ can be reduced. In fact, a canonical system of generators of $Perm$, contained in $G_n$'s, is given by the \textit{simple permutations}. 
Recall that a permutation $\sigma$ is called \textit{simple} if $\sigma$, as a map from $\{1,\dots,n\}$ to itself, does not map any non-trivial interval to another interval. For example, either one of the permutations listed in $G_4$, $G_5$ above is simple. The permutation $3156427$ is not simple, since it sends a non-trivial interval $\{3,4,5\}$ to $\{4,5,6\}$.
Simple permutations of length $n\geq 1$ are counted by the sequence A111111 (\cite{oeisA111111}):
\[
1, 2, 0, 2, 6, 46, 338, 2926, 28146, 298526\dots
\]
One may compare it with \eqref{Gcardseq}.
For $n\geq 4$, the number of simple permutations of length $n$ is given by ${s_n=-c_n+(-1)^{n+1}\cdot 2}$, where $c_n$ is the coefficient by $x^n$ of the functional inverse $F^{\langle-1\rangle}(x)$ of ${F(x)=\sum\limits_{n=1}^{\infty}n!x^n}$\cite{AlbertAtkinsonKlazar}.
In terms of percolative properties, simple permutations for $n\geq 4$ correspond to irreducible, in terms of the $(-\circ_i-)$-products in $Perm$, non-percolating sets of size $n$.
 Perhaps, one can make a case for treating them as the most extreme, in a sense to be made precise, non-percolating configurations of size $n$ in a $n$-by-$n$ grid. It could be interesting to study the statistics of such extreme configurations.
The suboperad of $Perm$ generated by $G_n$'s for $n\geq 4$ is free by \cite[Theorem 1]{AlbertAtkinsonKlazar}. The non-freeness of $Perm$ is only due to simple permutations $12$ and $21$ generating a copy of $As$ each.
This observation concludes our current study of $Perm$ in relation to bootstrap percolation. We hope to return to this topic later.

\subsection*{Acknowledgments}
The work is supported by RVO:67985840 of the Institute of Mathematics of the Czech Academy of Sciences and the Praemium Academiae grant of M. Markl. 

\printbibliography

\end{document}